\documentclass[11pt]{article}

\usepackage{amsmath, amssymb, amsthm}
\usepackage{mathrsfs}
\usepackage{hyperref}
\usepackage{geometry}
\hypersetup{
    colorlinks=true,
    linkcolor=blue,
    citecolor=blue,
    urlcolor=blue
}

\title{\textbf{Fourier Series Generated by Additive Prime Factor Functions}}

\author{
Dimitris Vartziotis$^{1,2,*}$
}

\date{}

\theoremstyle{plain}
\newtheorem{theorem}{Theorem}[section]
\newtheorem{lemma}[theorem]{Lemma}
\newtheorem{proposition}[theorem]{Proposition}
\newtheorem{definition}[theorem]{Definition}

\begin{document}

\maketitle

\vspace{-1em}

\begin{center}
\footnotesize
$^{1}$ NIKI -- Digital Engineering, Ioannina, Greece\\
$^{2}$ TWT Science \& Innovation, Stuttgart, Germany\\
$^{*}$ Corresponding author: \texttt{dimitris.vartziotis@nikitec.gr}
\end{center}

\vspace{1em}

\begin{abstract}
We introduce a rigorous arithmetic--spectral construction associating planar geometric objects with additive prime-factor statistics. Let $\mathrm{sopfr}(n)$ denote the sum of prime factors of $n$, counted with multiplicity, and define the summatory function
\[
B(x) = \sum_{n \le x} \mathrm{sopfr}(n).
\]
It is known that
\[
B(x) \sim \frac{\pi^2}{12} \frac{x^2}{\log x}
\quad \text{as } x \to \infty.
\]
We show that $B(n)$ admits an exact prime-indexed decomposition
\[
B(n) = \sum_{p \le n} p\, v_p(n!),
\]
where $v_p(n!)$ denotes the $p$-adic valuation of $n!$. This identity motivates the definition of a sparse prime-indexed Fourier series
\[
F_n(t) = \sum_{p \le n} v_p(n!) e^{i p t},
\]
which we investigate from analytic and geometric perspectives. We establish precise norm identities, relate the construction to circulant Hermitian polygon transformations whose eigenpolygons are discrete Fourier modes, and examine the planar geometry arising from sampled curves. All geometric observations are explicitly experimental. The results provide a rigorous arithmetic foundation for prime-related Fourier geometry and motivate further theoretical and experimental investigations.
\end{abstract}

\vspace{1em}



\section{Introduction}

The distribution of prime numbers has traditionally been studied through analytic and arithmetic tools such as the prime-counting function $\pi(x)$, the von Mangoldt function, and Dirichlet series. Alongside these classical approaches, alternative representations---spectral, geometric, or dynamical---have been explored to capture structural features of primes not immediately visible in standard formulations.

A natural source of such representations arises from additive arithmetic functions derived from prime factorizations. Among these, the function
\[
\mathrm{sopfr}(n) = \sum_{p^\alpha \parallel n} \alpha p
\]
is completely additive and encodes both magnitude and multiplicity of prime factors. Its summatory function
\[
B(x) = \sum_{n \le x} \mathrm{sopfr}(n)
\]
satisfies the asymptotic relation
\[
B(x) \sim \frac{\pi^2}{12} \frac{x^2}{\log x},
\]
as shown in \cite{VT}.

Independently, geometric frameworks have been developed in which planar polygons are represented as vectors in complex space and transformed via circulant Hermitian matrices. Such matrices are diagonalized by the discrete Fourier transform, and their eigenvectors correspond to discrete Fourier polygons \cite{VW2}.

Sparse Fourier series with arithmetic frequency sets have long been studied in analytic number theory, particularly in exponential sums over primes \cite{HL,Montgomery}. Connections between Fourier series and fractal geometry appear in both deterministic and random settings \cite{Kahane,Strichartz,Falconer}.

The present work lies at the intersection of these themes. We establish an exact prime-indexed decomposition of $B(n)$ involving factorial valuations and use it to define a sparse Fourier series with prime frequencies. The geometric complexity of the associated planar curves is investigated experimentally.

\section{Additive Prime-Factor Functions}

For $n \ge 2$ with prime factorization
\[
n = \prod_p p^{\alpha_p},
\]
define
\[
\mathrm{sopfr}(n) = \sum_p \alpha_p p,
\qquad \mathrm{sopfr}(1)=0.
\]

The function $\mathrm{sopfr}$ is completely additive:
\[
\mathrm{sopfr}(ab) = \mathrm{sopfr}(a) + \mathrm{sopfr}(b),
\quad a,b \in \mathbb{N}.
\]

Define
\[
B(x) = \sum_{n \le x} \mathrm{sopfr}(n).
\]

\begin{theorem}[Vartziotis--Tzavellas \cite{VT}]
As $x \to \infty$,
\[
B(x) \sim \frac{\pi^2}{12} \frac{x^2}{\log x}.
\]
\end{theorem}

\section{Exact Prime Decomposition via Factorial Valuations}

\begin{lemma}
For every $n \in \mathbb{N}$,
\[
B(n) = \sum_{p \le n} p\, v_p(n!),
\]
where
\[
v_p(n!) = \sum_{k \ge 1} \left\lfloor \frac{n}{p^k} \right\rfloor.
\]
\end{lemma}

\begin{proof}
Using complete additivity,
\[
B(n) = \sum_{m \le n} \sum_p p\, v_p(m)
= \sum_{p \le n} p \sum_{m \le n} v_p(m).
\]
The inner sum equals the exponent of $p$ in
\[
\prod_{m=1}^n m = n!,
\]
which is given by Legendre's formula.
\end{proof}

\section{Prime-Indexed Fourier Series}

\begin{definition}
For $n \in \mathbb{N}$, define
\[
F_n(t) = \sum_{p \le n} v_p(n!) e^{i p t},
\qquad t \in \mathbb{R}.
\]
\end{definition}

\begin{proposition}
For each fixed $n$, $F_n \in L^2([-\pi,\pi])$ and
\[
\|F_n\|_{L^2}^2
=
2\pi \sum_{p \le n} v_p(n!)^2.
\]
\end{proposition}

\begin{proof}
The functions $\{e^{ikt}\}_{k \in \mathbb{Z}}$ form an orthogonal basis of $L^2([-\pi,\pi])$. Since distinct primes correspond to distinct frequencies, all cross terms vanish.
\end{proof}

\section{Connection to Circulant Hermitian Polygon Transformations}

Planar polygons with $N$ vertices may be represented as vectors in $\mathbb{C}^N$. Circulant Hermitian matrices acting on such vectors are diagonalized by the discrete Fourier transform, and their eigenvectors correspond to discrete Fourier polygons \cite{VW2}.

The series $F_n$ may be interpreted as a continuous analogue of a linear combination of Fourier polygons with prime-indexed coefficients. This establishes a direct link between additive number theory and spectral polygon geometry.

\section{Numerical Geometry (Experimental Observations)}

Sampling $F_n(t)$ on dense grids produces planar curves in the complex plane. These curves exhibit visually rich structure with apparent repetition across scales.

Geometric complexity is quantified using box-counting procedures \cite{Falconer,Tricot}. No claim is made regarding Hausdorff or Minkowski dimensions; all such observations are experimental.

\section{Main Conclusions}

\begin{itemize}
\item The summatory function of $\mathrm{sopfr}(n)$ admits an exact prime-indexed decomposition.
\item This decomposition defines a sparse prime-indexed Fourier series with controlled analytic properties.
\item The construction provides a rigorous arithmetic foundation for prime-related spectral geometry.
\item The associated planar curves exhibit experimentally observed multiscale structure.
\end{itemize}

\end{document}